\documentclass[11pt]{article}
\usepackage{amsfonts}
\usepackage{latexsym}
\usepackage{cite}
\usepackage{amsmath,amsfonts,latexsym,amssymb}
\usepackage[mathscr]{eucal}
\usepackage{cases}
\usepackage{amsthm}
\usepackage[bf,small]{caption2}
\usepackage{float}
\usepackage{graphicx}
\usepackage{amssymb}
\usepackage[all]{xy}
\usepackage{slashbox,multirow}

\newtheorem{theorem}{theorem}[section]

\newtheorem{exmp}[theorem]{Example}
\newtheorem{lem}[theorem]{Lemma}

\newtheorem{prop}[theorem]{Proposition}

\newtheorem{rmk}[theorem]{Remark}
\newtheorem{thm}[theorem]{Theorem}

\begin{document}

\title{\textbf{Trace-free ${\rm SL}(2,\mathbb{C})$-representations of arborescent links}}
\author{\Large Haimiao Chen
\footnote{Email: \emph{chenhm@math.pku.edu.cn}} \\
\normalsize \em{Beijing Technology and Business University, Beijing, China}}
\date{}
\maketitle

\begin{abstract}
  Given a link $L\subset S^3$, a representation $\pi_1(S^3-L)\to{\rm SL}(2,\mathbb{C})$ is {\it trace-free} if it sends each meridian to an element with trace zero. We present a method for completely determining trace-free ${\rm SL}(2,\mathbb{C})$-representations for arborescent links. Concrete computations are done for a class of 3-bridge arborescent links.

\end{abstract}

\section{Introduction}

Let $G$ be a linear group. Given a link $L\subset S^{3}$, a {\it trace-free $G$-representation} is a homomorphism $\rho:\pi_{1}(S^{3}-L)\to G$ which sends each meridian to an element with trace zero.

Researchers have paid attention to such representations for long.
For a knot $K$, Lin \cite{Li92} defined an invariant $h(K)$ roughly by counting (with signs) conjugacy classes of trace-free ${\rm SU}(2)$-representations of $K$, and showed that $h(K)$ equals half of the signature of $K$.
Interestingly, it was shown in \cite{Le14} that if $L$ is an alternating link or a 2-component link, then its Khovanov homology is isomorphic to the singular homology of the space of binary dihedral representations (a special kind of trace-free ${\rm SU}(2)$-representations) of $L$, as graded abelian groups.
Also interesting is that, by the result of \cite{NY12}, each trace-free ${\rm SL}(2,\mathbb{C})$-representation of $L$ gives rise to a representation $\pi_1(M_2(L))\to{\rm SL}(2,\mathbb{C})$, where $M_2(L)$ is the double covering of $S^3$ branched along $L$.

As is well-known, linear representations of links are useful, but there is no general method to systematically find nontrivial ones.
We may first focus on trace-free representations, which turn out to be easier to manipulate. This makes up another motivation.
In this paper we aim to determine all trace-free ${\rm SL}(2,\mathbb{C})$-representations for each arboresent link, which is a continuation of the previous work \cite{Ch16}.
From now on, we abbreviate ``trace-free ${\rm SL}(2,\mathbb{C})$-representation" to ``representation".
The main results are Theorem \ref{thm:key} and Theorem \ref{thm:method}. Based on them, we are able to explicitly determine all representations for any given arborescent link.

We introduce some basic notions in Section 2, and define representation of a tangle in Section 3. In the main body, Section 4, we clarify properties of representations of arborescent tangles, uncovering an exquisite structure in the representation space, and then explain how to determine all representations for arborescent tangles. These lead to a practical method to find all representations for arborescent links, as presented in Section 5. As an illustration, explicit formulas are given for a class of 3-bridge arborescent links.

It should be pointed out that most of the results in this paper will remain valid when $\mathbb{C}$ is replaced by a general field, or even a ring.

\section{Preliminary}

\subsection{Tangles and tangle diagrams}

Adopt the notions in \cite{KL03}.
A {\it 2-tangle} $T$ is an embedding of $[0,1]\sqcup[0,1]\sqcup mS^1$ (where $mS^1$ denotes the disjoint union of $m$ copies of $S^1$ and $m\ge 0$) into a 3-dimensional ball $B^3$, such that $\partial T\cap B^3$ is a specific
set of four points on $\partial B^3=S^2$.
Two 2-tangles $T, T'$ are {\it isotopic} if there exists a homeomorphism $h:B^3\to B^3$ such that $h|_{S^2}={\rm id}$ and $h(T)=T'$.

Each tangle can be represented by a diagram by a usual manner, as done throughout the paper. Abusing the notation, by ``tangle" we also mean a diagram, whenever there is no ambiguity.

\begin{figure} [h]
  \centering
  \includegraphics[width=0.8\textwidth]{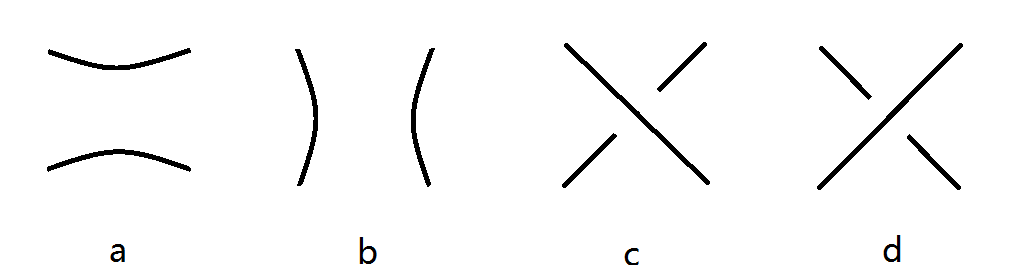}\\
  \caption{The simplest four tangles: (a) $[0]$, (b) $[\infty]$, (c) $[1]$, (d) $[-1]$}\label{fig:basic}
\end{figure}

\begin{figure}
  \centering
  \includegraphics[width=0.63\textwidth]{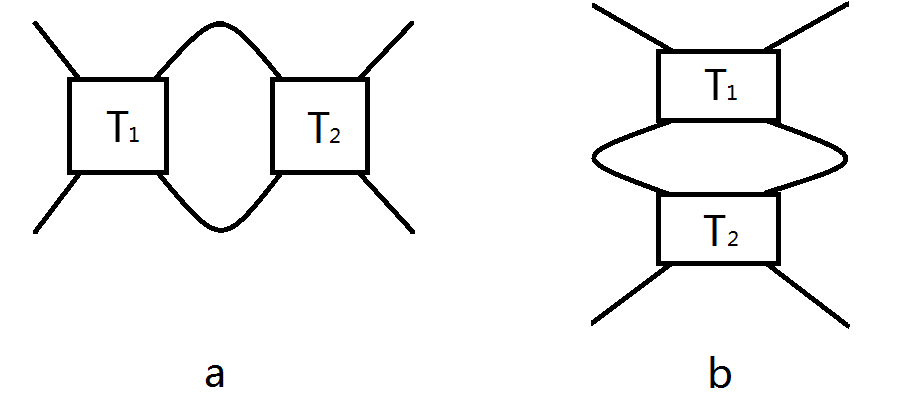}\\
  \caption{(a) $T_{1}+T_{2}$; (b) $T_{1}\ast T_{2}$} \label{fig:composition}
\end{figure}

Let $\mathcal{T}_{2}$ denote the set of 2-tangles.
The simplest four ones are given in Figure \ref{fig:basic}. In $\mathcal{T}_{2}$ there are two binary operations: {\it horizontal composition} $+$ and {\it vertical composition} $\ast$, as illustrated in Figure \ref{fig:composition}.
Let $\mathcal{T}_{\textrm{ar}}$ denote the subset of $\mathcal{T}_{2}$ containing $[\pm 1]$ and closed under $+$ and $\ast$.
An element of $\mathcal{T}_{\textrm{ar}}$ is called an {\it arborescent tangle}. Each arborescent tangle other than $[0],[\infty]$ can be obtained from copies of $[\pm 1]$ via iterated horizontal and vertical compositions.

\begin{lem}
There is a unique function $f:\mathcal{T}_2\to\mathbb{Q}\cup\{\infty\}$ such that $f(T)$ depends only on the isotopy class of $T$, and characterized by
\begin{align}
f([\pm 1])&=\pm 1, \label{eq:f(pm1)} \\
f(T_{1}+T_{2})&=f(T_{1})+f(T_{2}), \label{eq:f1} \\
f(T_{1}\ast T_{2})&=\frac{1}{1/f(T_{1})+1/f(T_{2})}. \label{eq:f2}
\end{align}
As a convention, $1/0=\infty$, $1/\infty=0$, $a+b=\infty$ if $a=\infty$ or $b=\infty$.
\end{lem}
This was established in \cite{KL03}. Call $f(T)$ the {\it fraction} of $T$.

For $k\ne 0$, the horizontal composite of $|k|$ copies of $[1]$ (resp. $[-1]$) is denoted by $[k]$ if $k>0$ (resp. $k<0$), and the vertical composite of $|k|$ copies of $[1]$ (resp. $[-1]$) is denoted by $[1/k]$ if $k>0$ (resp. $k<0$). Obviously, $f([k])=k$ and $f([1/k])=1/k$.
Given integers $k_{1},\ldots,k_{m}$, the \emph{rational tangle} $[[k_{1}],\ldots,[k_{m}]]$ is defined as
\begin{align}
\begin{cases}
[k_{1}]\ast [1/k_{2}]+\cdots\ast [1/k_{m}], &\text{if\ \ } 2\mid m, \\
[k_{1}]\ast [1/k_{2}]+\cdots +[k_{m}], &\text{if\ \ } 2\nmid m.
\end{cases}
\end{align}
It is easy to see that
\begin{align}
f([[k_{1}],\ldots,[k_{m}]])=[[k_{1},\ldots,k_{m}]]^{(-1)^{m-1}},
\end{align}
where the \emph{continued fraction} $[[k_{1},\ldots,k_{m}]]\in\mathbb{Q}$ is defined inductively as
\begin{align}
[[k_{1}]]=k_{1}, \hspace{10mm}
[[k_{1},\ldots,k_{m}]]=k_{m}+1/[[k_{1},\ldots,k_{m-1}]].
\end{align}
By Theorem 3 of \cite{KL04}, two rational tangles are isotopic if and only if their fractions coincide.
Denote $[[k_{1}],\ldots,[k_{m}]]$ as $[p/q]$ if its fraction equals $p/q$.

A {\it Montesinos tangle} is a tangle of the form $[p_1/q_1]\ast\cdots\ast[p_n/q_n]$.

\begin{figure} [h]
  \centering
  \includegraphics[width=0.56\textwidth]{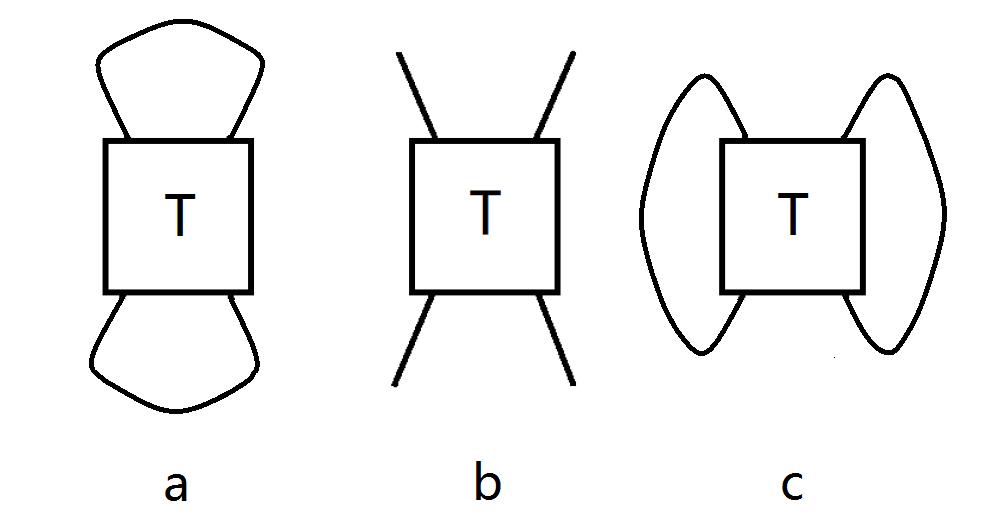}\\
  \caption{(a) the numerator of $T$; (b) a tangle $T$; (c) the denominator of $T$}   \label{fig:D-N}
\end{figure}

Each tangle $T\in\mathcal{T}_{2}$ leads to two links: the {\it numerator} $N(T)$ and the {\it denominator} $D(T)$; see Figure \ref{fig:D-N}.
An {\it arborescent link} is a link of the form $N(T)$ or $D(T)$, with $T\in\mathcal{T}_{\rm ar}$.
In particular, a {\it Montesinos link} is a link of the form $D([p_1/q_1]\ast\cdots\ast[p_n/q_n])$, denoted as $M(p_1/q_1,\ldots,p_n/q_n)$.

\subsection{Some linear algebra}

Let ${\rm SL}^{0}(2,\mathbb{C})$ denote the subset of ${\rm SL}(2,\mathbb{C})$ consisting of matrices with trace zero.
For any $X,Y\in{\rm SL}^{0}(2,\mathbb{C})$, we have
\begin{align}
X^{-1}=-X,  \qquad XYX^{-1}=-{\rm tr}(XY)\cdot X-Y.   \label{eq:basic}
\end{align}

For $X,Z\in{\rm SL}(2,\mathbb{C})$, denote $ZXZ^{-1}$ by $Z.X$. For $X,Y,X',Y'\in{\rm SL}^0(2,\mathbb{C})$, denote $(X,Y)\sim(X',Y')$ if there exists $Z\in{\rm SL}(2,\mathbb{C})$ such that $Z.X=X'$ and $Z.Y=Y'$.
This defines an equivalence relation on ${\rm SL}^0(2,\mathbb{C})\times{\rm SL}^0(2,\mathbb{C})$.

For $a,b,c\in\mathbb{C}$ with $a,b\ne 0$, let
\begin{align}
t(a)&=a+a^{-1}, & &\ \ s(a)=a-a^{-1}, \\
A_b(a)&=\frac{1}{2}\left(\begin{array}{cc} t(a)i & s(a)b \\ s(a)b^{-1} & -t(a)i \\ \end{array}\right), && \ \\
B_1(c)&=\left(\begin{array}{cc} i & c \\ 0 & -i \end{array}\right), &&
B_2(c)=\left(\begin{array}{cc} i & 0 \\ c & -i \end{array}\right).
\end{align}
Let $A=A_1(1)$ and $B_\ell=B_\ell(1)$ for $\ell=1,2$, i.e.,
\begin{align}
A=\left(\begin{array}{cc} i & 0 \\ 0 & -i \end{array}\right), \qquad
B_1=\left(\begin{array}{cc} i & 1 \\ 0 & -i \end{array}\right), \qquad
B_2=\left(\begin{array}{cc} i & 0 \\ 1 & -i \end{array}\right).
\end{align}

As can be verified,
\begin{align}
{\rm tr}(A_1(a_1)A_c(a_2))&=\frac{1}{4}s(a_1)s(a_2)t(c)-\frac{1}{2}t(a_1)t(a_2)=:\gamma_{a_1,a_2}(c), \label{eq:formula-AA} \\
{\rm tr}(A_c(a)B_\ell)&=\frac{1}{2}s(a)c^{(-1)^\ell}-t(a)=:\delta_{a,\ell}(c). \label{eq:formula-AB}
\end{align}
These two formulas will be useful in Example \ref{eg:2-rational} and the last section.

\begin{rmk}
\rm As a special case of (\ref{eq:formula-AA}), ${\rm tr}(AA_1(a))=-t(a)$. Observe that if $a\notin\{\pm 1\}$ and ${\rm tr}(AX)=-t(a)$, then $X=A_b(a)$ for some $b$.
\end{rmk}

For any $a,b\in\mathbb{C}$, we have
\begin{align}
B_\ell(a).B_\ell(b)=B_\ell(2a-b);  \label{eq:B-B}
\end{align}
in particular, $B_\ell(a).A=B_\ell(2a)$.

For $X,Y\in{\rm SL}^{0}(2,\mathbb{C})$ with ${\rm tr}(XY)=t(a)$, call the pair $(X,Y)$ {\it regular} ({\it R} for short) if $(X,Y)\sim (A,-A_1(a))$, and call it {\it non-regular} ({\it NR} for short) otherwise. It is a simple exercise in linear algebra that $(X,Y)$ is NR if and only if $a\in\{\pm 1\}$ and $X\ne -aY$, in which case $(X,Y)\sim (A,-aB_1)$ or $(X,Y)\sim(A,-aB_2)$;
it is clear that $(A,-aB_1)\not\sim (A,-aB_2)$.
When $(X,Y)\sim (A,-aB_\ell)$, say $(X,Y)$ is NR of {\it type $\ell$} (NR$_\ell$ for short).

For $(a,b)\in\mathbb{C}^\times\times\mathbb{C}^\times$ with $a\notin\{\pm 1\}$, put
\begin{align}
S_{(a,b)}=\frac{1}{s(a)}\left(\begin{array}{cc} s(ab^{-1}) & s(b) \\ -s(b) &  s(ab) \\ \end{array}\right).
\end{align}
Observe that
\begin{align}
S_{(a,b)}=\pm I&\Leftrightarrow b=\pm 1, \label{eq:S1} \\
S_{(a,b)}=S_{(a',b')}&\Leftrightarrow(a,b)\sim(a',b'), \label{eq:S2} \\
S_{(a,b)}S_{(a,b')}&=S_{(a,bb')}. \label{eq:S3}
\end{align}

Let
\begin{align}
\mathcal{U}=\mathbb{C}^\times\times\mathbb{C}^\times/(a,b)\sim(a^{-1},b^{-1}).
\end{align}
Denote the equivalence class of $(a,b)$ by $[a,b]$. From (\ref{eq:S2}) we see that $S_{[a,b]}$ can be defined without ambiguity.
Also well-defined is the map
\begin{align}
t:\mathcal{U}\to\mathbb{C}^2, \qquad [a,b]\mapsto (t(a),t(b)).
\end{align}

For $c\in\{\pm 1\}$ and $u\in\mathbb{C}$, put
\begin{align}
C_c(u)=\left(\begin{array}{cc} 1+u & -cu \\ cu & 1-u \\ \end{array}\right).
\end{align}
It is easy to see that for any $u_1,u_2$,
\begin{align}
C_c(u_1)C_c(u_2)=C_c(u_1+u_2), \label{eq:C}
\end{align}
and for any $a,b,u\in\mathbb{C}$,
\begin{align}
(B_\ell(a),-cB_\ell(b))C_c(u)=(B_\ell(a+ua-ub), -cB_\ell(b+ua-ub)). \label{eq:C-B}
\end{align}

\section{Representations of tangles}

In a tangle $T$, each arc together with a choice of direction is called a {\it directed arc}; let $\mathcal{D}(T)$ denote the set of directed arcs of $T$. For $x\in\mathcal{D}(T)$, let $x^{-1}$ denote the same underlying arc of $x$ with direction reversed.

Given a 2-tangle $T$.
A {\it (trace-free ${\rm SL}(2,\mathbb{C})$-)representation} of $T$ is a homomorphism $\pi_1(B^3-T)\to{\rm SL}(2,\mathbb{C})$ sending each meridian to an element of ${\rm SL}^{0}(2,\mathbb{C})$.
In virtue of Wirtinger presentation, it is the same as a map $\rho:\mathcal{D}(T)\to {\rm SL}^{0}(2,\mathbb{C})$ such that $\rho(x^{-1})=\rho(x)^{-1}=-\rho(x)$ for each $x$ and
\begin{align}
\rho(z)=\rho(x)\rho(y)\rho(x)^{-1}=-{\rm tr}(\rho(x)\rho(y))\cdot\rho(x)-\rho(y) \label{eq:rep}
\end{align}
for all $x,y,z$ that are placed as in Figure \ref{fig:crossing} (a) or (b).
\begin{figure} [h]
  \centering
  \includegraphics[width=0.62\textwidth]{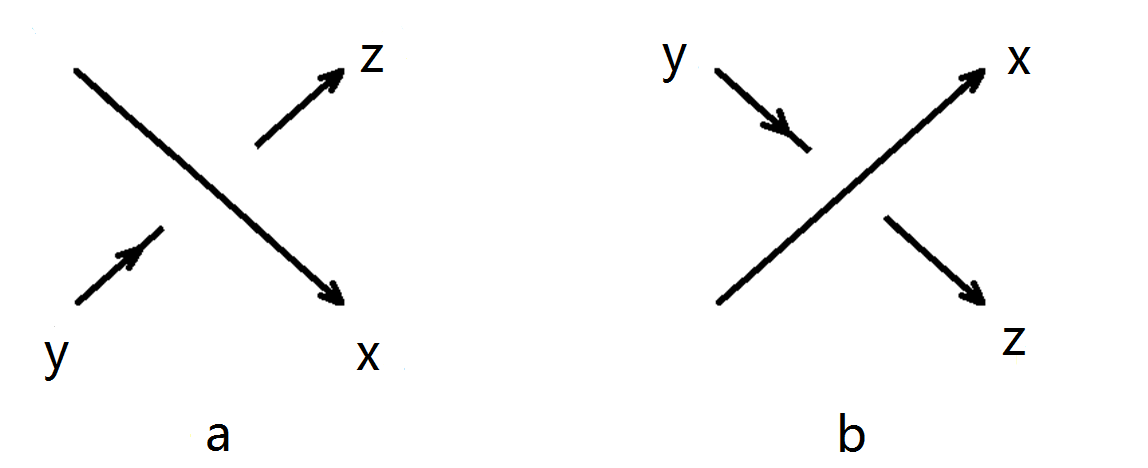}\\
  \caption{Three arcs forming a crossing} \label{fig:crossing}
\end{figure}
To denote such a map, we choose for each arc a direction and label it with an element of ${\rm SL}^{0}(2,\mathbb{C})$.
Let $\mathcal{R}(T)$ denote the set of representations of a tangle $T$.

\begin{rmk} \label{rmk:involution}
\rm Observe that, for each $\rho\in\mathcal{R}(T)$, the function
\begin{align}
\mathcal{D}(T)\to{\rm SL}^0(2,\mathbb{C}), \qquad x\mapsto \rho(x)^t
\end{align}
($X^t$ denotes the transpose of $X$) is also a representation. This defines an involution on $\mathcal{R}(T)$.
\end{rmk}

Let $T^{{\rm nw}}, T^{{\rm ne}}, T^{{\rm sw}}, T^{{\rm se}}$ denote the directed arcs shown in Figure \ref{fig:T}.
Given $\rho\in\mathcal{R}(T)$, let
\begin{align}
\rho^{{\rm nw}}=\rho(T^{{\rm nw}}), \quad \rho^{{\rm sw}}=\rho(T^{{\rm sw}}), \quad
&\rho^{{\rm ne}}=\rho(T^{{\rm ne}}), \quad \rho^{{\rm se}}=\rho(T^{{\rm se}}), \\
\rho^w=(\rho^{\rm nw},\rho^{\rm sw}), \quad \rho^e=(\rho^{\rm ne},\rho^{\rm se}), \quad
&\rho^n=(\rho^{\rm nw},\rho^{\rm ne}), \quad \rho^s=(\rho^{\rm sw},\rho^{\rm se}), \\
{\rm tr}_{v}(\rho)={\rm tr}(\rho^{{\rm nw}}\rho^{{\rm sw}})&={\rm tr}(\rho^{{\rm ne}}\rho^{{\rm se}}), \\
{\rm tr}_{h}(\rho)={\rm tr}(\rho^{{\rm nw}}\rho^{{\rm ne}})&={\rm tr}(\rho^{{\rm sw}}\rho^{{\rm se}}).
\end{align}
Say that $\rho^n$ is R (resp. NR$_\ell$) if $(\rho^{\rm nw},\rho^{\rm ne})$ is R (resp. NR$_\ell$), and so on.

\begin{figure} [h]
  \centering
  \includegraphics[width=0.3\textwidth]{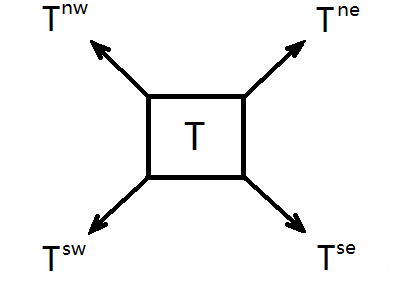}\\
  \caption{A tangle diagram with the ends directed outwards} \label{fig:T}
\end{figure}

Let $\rho$ be a representation of $T$. Call $\rho$ {\it reducible} if all of the elements in its image ${\rm Im}(\rho)$ share an eigenvector, or equivalently, $\rho$ is conjugate to an {\it upper-triangular} representation, which by definition sends each $x\in\mathcal{D}(T)$ to an upper-triangular matrix; it follows that ${\rm tr}(\rho(x)\rho(y))\in\{\pm 2\}$ for any $x,y\in\mathcal{D}(T)$.
In particular, call $\rho$ {\it abelian} if ${\rm Im}(\rho)$ is abelian.
Call $\rho$ {\it irreducible} if it is not reducible.

If $T'$ is a sub-tangle of $T$, then the map
\begin{align}
\rho|_{T'}:\mathcal{D}(T')\hookrightarrow \mathcal{D}(T)\stackrel{\rho}\longrightarrow{\rm SL}^0(2,\mathbb{C})
\end{align}
is a representation, called the {\it restriction} of $\rho$ on $T'$.

Given $\rho_j\in\mathcal{R}(T_j), j=1,2$. If $\rho_2^w=-\rho_1^e$, meaning $\rho_2^{\rm nw}=-\rho_1^{\rm ne}$, $\rho_2^{\rm sw}=-\rho_1^{\rm se}$, then we can ``glue" $\rho_1$ and $\rho_2$ to obtain a representation of $T_1+T_2$, denoted by $\rho_1+\rho_2$. Similarly, if $\rho_2^n=-\rho_1^s$, then we can define $\rho_1\ast\rho_2\in\mathcal{R}(T_1\ast T_2)$.
Thus there are partial compositions
\begin{align}
+&:\ \mathcal{R}(T_1)\times_h\mathcal{R}(T_2)\to\mathcal{R}(T_1+T_2), && (\rho_1,\rho_2)\mapsto \rho_1+\rho_2, \\
\ast&:\ \mathcal{R}(T_1)\times_v\mathcal{R}(T_2)\to\mathcal{R}(T_1\ast T_2), && (\rho_1,\rho_2)\mapsto \rho_1\ast\rho_2,
\end{align}
where
\begin{align}
\mathcal{R}(T_1)\times_h\mathcal{R}(T_2)&=\{(\rho_1,\rho_2)\in\mathcal{R}(T_1)\times\mathcal{R}(T_1)\colon \rho_2^w=-\rho_1^e\}, \\
\mathcal{R}(T_1)\times_v\mathcal{R}(T_2)&=\{(\rho_1,\rho_2)\in\mathcal{R}(T_1)\times\mathcal{R}(T_1)\colon \rho_2^n=-\rho_1^s\}.
\end{align}

\begin{exmp} \label{exmp:pm1}
\rm Representations of $[\pm 1]$ are easy to describe.
If $\rho$ is a representation of $[1]$ (resp. $[-1]$), with $\rho^{w}=(X,Y)$ and ${\rm tr}(XY)=t(a)$, then
$\rho^e=(-t(a)X-Y,X)$ (resp. $\rho^e=(Y,-t(a)Y-X)$).
Note that
\begin{enumerate}
  \item[\rm(a)] $\rho$ is irreducible if and only if $a\notin\{\pm 1\}$;
  \item[\rm(b)] when $a\in\{\pm 1\}$,
      \begin{align*}
      \rho^w\ \text{is\ NR}\ \Leftrightarrow\ \rho^e\ \text{is\ NR}\ \Leftrightarrow\ Y\ne-aX\ \Leftrightarrow\ \rho^n\ \text{is\ NR}\ \Leftrightarrow\ \rho^s\text{\ is\ NR}.
      \end{align*}
\end{enumerate}

\end{exmp}

\begin{figure} [h]
  \centering
  \includegraphics[width=0.6\textwidth]{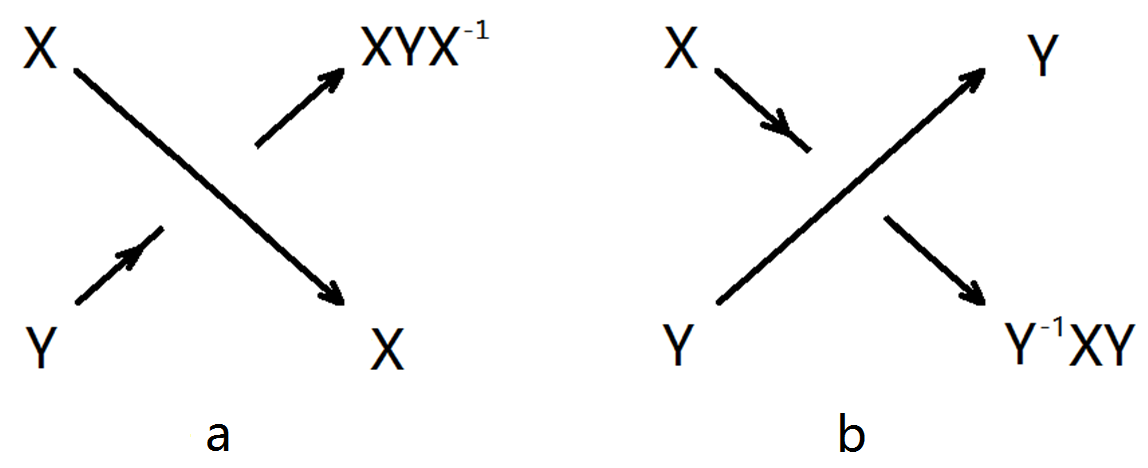} \\
  \caption{(a) a representation of $[1]$;\ (b) a representation of $[-1]$} \label{fig:rep-pm1}
\end{figure}

\section{Representation spaces of arborescent tangles}

We separately investigate reducible representations.

\begin{prop} \label{prop:reducible}
Suppose $T\in\mathcal{T}_{\rm ar}$, with $f(T)=f$; suppose $\rho\in\mathcal{R}(T)$ is reducible, so that ${\rm tr}_v(\rho)=2a,{\rm tr}_h(\rho)=2b$ for some $a,b\in\{\pm 1\}$. Then
\begin{itemize}
  \item[\rm(i)] $\rho^e=-b\rho^wC_a(f)$ if $f\ne\infty$,
  \item[\rm(ii)] $\rho^s=-a\rho^nC_b(f^{-1})$ if $f\ne 0$.
\end{itemize}
\end{prop}

\begin{proof}
Referring to (b) in Example \ref{exmp:pm1}, it is easy to check (i), (ii) for $[\pm 1]$. Suppose the conclusion is true for $T_1, T_2$, we shall prove it for $T=T_1+T_2$; the proof for $T=T_1\ast T_2$ is similar.  Let $f_j=f(T_j)$, then $f=f_1+f_2$ by (\ref{eq:f1}). Let $\rho_j=\rho|_{T_j}$, then $\rho_1,\rho_2$ are both reducible. Clearly, ${\rm tr}_v(\rho_1)={\rm tr}_v(\rho_2)=2a$; suppose ${\rm tr}_h(\rho_j)=2b_j$ with $b_j\in\{\pm 1\}$.

If $f\ne\infty$, then $f_1,f_2\ne\infty$, so $\rho_j^e=-b_j\rho_j^wC_{a}(f_j)$, $j=1,2$. Hence
\begin{align*}
\rho^e&=\rho_2^e=-b_2\rho_2^wC_a(f_2)=b_2\rho_1^eC_a(f_2)=-b_1b_2\rho_1^wC_a(f_1)C_a(f_2) \\
&=-b_1b_2\rho^wC_a(f_1+f_2)=-b'\rho^wC_a(f), \qquad \text{with} \qquad b'=b_1b_2,
\end{align*}
where (\ref{eq:C}) is used. Computing directly, we obtain
\begin{align}
{\rm tr}(\rho^{\rm ne}\rho^{\rm nw})={\rm tr}(-b'((1+f)\rho^{\rm nw}+af\rho^{\rm sw})\rho^{\rm nw})=b'(2(1+f)-2a^2f)=2b',
\label{eq:deduce-beta}
\end{align}
which shows that actually $b=b'$. Hence (i) is true.

If $f\ne 0,\infty$, then writing $\rho^{\rm sw}=-a\rho^{\rm nw}+X$, so that
$$\rho^e=-b\rho^wC_a(f)=-b(\rho^{\rm nw}+afX,-a\rho^{\rm nw}+(1-f)X),$$
we obtain $\rho^s=-a\rho^nC_b(f^{-1})$, i.e., (ii) is true in the case $f\ne\infty$.

If $f=\infty$, then $f_1=\infty$ or $f_2=\infty$.
\begin{itemize}
  \item If $f_1=f_2=\infty$, then $\rho_j^s=-a\rho_j^nC_{b_j}(0)$, so that $\rho_1^{\rm sw}=-a\rho_1^{\rm nw}$ and $\rho_2^{\rm se}=-a\rho_2^{\rm ne}$, i.e., $\rho^s=-a\rho^nC_a(0)$;
  \item if $f_1=\infty$ and $f_2\ne\infty$, then $\rho_1^s=-a\rho_1^n$ and $\rho_2^e=-b_2\rho_2^wC_a(f_2)$, which implies $\rho_2^{\rm se}=-a\rho_2^{\rm ne}$, hence also $\rho^s=-a\rho^nC_a(0)$;
  \item if $f_1\ne \infty$ and $f_2=\infty$, the proof is similar.
\end{itemize}
Thus (ii) is true in the case $f=\infty$.
\end{proof}

Now state and prove our first main result.

\begin{thm} \label{thm:key}
Let $T\in\mathcal{T}_{\rm ar}$, with $f(T)=f$, and let $\rho\in\mathcal{R}(T)$.
\begin{enumerate}
  \item[\rm(i)] Let $\{\alpha,\beta\}=\{e,w\}$. If $\rho^{\alpha}$ is NR$_\ell$, then ${\rm tr}_v(\rho)=2a$, ${\rm tr}_h(\rho)=2b$ with $a,b\in\{\pm 1\}$, $\rho^{\beta}=-b\rho^{\alpha}C_a(u)$ for some $u\in\mathbb{Q}$, and $\rho^{\beta}$ is also NR$_\ell$; moreover, $u\ne 0$ if and only if $\rho^n, \rho^s$ are NR$_\ell$.
  \item[\rm(ii)] Let $\{\alpha,\beta\}=\{n,s\}$. If $\rho^{\alpha}$ is NR$_\ell$, then ${\rm tr}_v(\rho)=2a$, ${\rm tr}_h(\rho)=2b$ with $a,b\in\{\pm 1\}$, $\rho^{\beta}=-a\rho^{\alpha}C_b(v)$ for some $v\in\mathbb{Q}$, and $\rho^{\beta}$ is also NR$_\ell$; moreover, $v\ne 0$ if and only if $\rho^w, \rho^e$ are NR$_\ell$.
  \item[\rm(iii)] If $t_v\notin\{\pm 2\}$, then $\rho^n,\rho^{w}, \rho^s, \rho^e$ are all R, and there exists a unique $[a,b]\in\mathcal{U}$ with $t([a,b])=({\rm tr}_v(\rho),{\rm tr}_h(\rho))$ and $\rho^e=-\rho^wS_{[a,b]}$.
  \item[\rm(iv)] If $t_h\notin\{\pm 2\}$, then $\rho^n,\rho^{w}, \rho^s, \rho^e$ are all R, and there exists a unique $[a,b]\in\mathcal{U}$ with $t([a,b])=({\rm tr}_v(\rho),{\rm tr}_h(\rho))$ and $\rho^s=-\rho^nS_{[b,a]}$.
\end{enumerate}

In any case, denote $[a,b]\in\mathcal{U}$ by $\chi(\rho)$.
\end{thm}

\begin{proof}
For $T\in\mathcal{T}_{\rm ar}$, let $H_{{\rm k}}(T)$ denote the statement (k) for $T$.
It is easy to verify $H_{{\rm k}}([\pm 1]), {\rm k}\in\{{\rm i},\ldots,{\rm iv}\}$.
Suppose $H_{{\rm k}}(T_j), {\rm k}\in\{{\rm i},\ldots,{\rm iv}\}, j\in\{1,2\}$ have been established. We shall prove $H_{{\rm k}}(T), {\rm k}\in\{{\rm i},\ldots,{\rm iv}\}$ for $T=T_1+T_2$; the proof for $T_1\ast T_2$ is similar.
Let $\rho_j=\rho|_{T_j}$.

(i) Assume $\alpha=w, \beta=e$; the proof for the other case is similar.

For $j=1,2$, by $H_{\rm i}(T_j)$, ${\rm tr}_v(\rho_j)=2a, {\rm tr}_h(\rho_j)=2b_j$ with $a,b_j\in\{\pm 1\}$, and there exists $u_j\in\mathbb{Q}$, such that $\rho_j^e=-b_j\rho_j^wC_a(u_j)$. Then due to $\rho_2^w=-\rho_1^e$ and (\ref{eq:C}), we have
$$\rho^e=-(b_1b_2)\rho_1^wC_a(u_1)C_a(u_2)=-b'\rho^wC_a(u),$$
with $u=u_1+u_2$, $b'=b_1b_2$.
Computing similarly as in (\ref{eq:deduce-beta}),
\begin{align*}
{\rm tr}(\rho^{\rm ne}\rho^{\rm nw})={\rm tr}(-b((1+u)\rho^{\rm nw}+au\rho^{\rm sw})\rho^{\rm nw})=b'(2(1+u)-2a^2u)=2b'.
\end{align*}
Hence $b=b'$, and $\rho^e=-b'\rho^wC_a(u)$.

Up to conjugacy we may assume $\rho^w=(A,-aB_\ell)$, then
$$\rho^e=-b(A,-aB_\ell) C_a(u)=-b(B_\ell(-u),-aB_\ell(1-u)).$$
Since by (\ref{eq:B-B}), $B_\ell(-u)=B_\ell(-u/2).A$ and $B_\ell(1-u)=B_\ell(-u/2).B_\ell(-1)$, we see that $\rho^e$ is NR$_\ell$. Obviously,
$u\ne 0$ if and only if $\rho^n,\rho^s$ are NR$_\ell$.

(iii) If ${\rm tr}_v(\rho)\notin\{\pm 2\}$, then $\rho^w$, $\rho^e$ are R; by $H_{\rm iii}(T_j)$, there exists a unique $[a_j,b_j]\in\mathcal{U}$ such that $t([a_j,b_j])=({\rm tr}_v(\rho_j),{\rm tr}_h(\rho_j))$ and $\rho_j^e=-\rho_j^wS_{[a_j,b_j]}$. Replacing $a_2$ by its inverse if necessary, we may assume $a_2=a_1=:a$. Then
\begin{align}
\rho^e=-\rho_2^wS_{[a,b_2]}=\rho_1^eS_{[a,b_2]}=-\rho_1^wS_{[a,b_1]}S_{[a,b_2]}=-\rho^wS_{[a,b]}, \label{eq:e-w}
\end{align}
with $b=b_1b_2$.
By computation, ${\rm tr}(\rho^{\rm nw}\rho^{\rm ne})=t(b)$.
If $b\in\{\pm 1\}$, then by (\ref{eq:S1}), $\rho^e=-b\rho^w$. Thus $\rho^n,\rho^s$ are R.

(iv) Suppose ${\rm tr}_h(\rho)\notin\{\pm 2\}$. Then $\rho^n, \rho^s$ are R, and by (i), $\rho^w, \rho^e$ are R. If ${\rm tr}_v(\rho)=2a$ with $a\in\{\pm 1\}$, then $\rho^s=-a\rho^n=-a\rho^nS_{[b,a]}$, where $b$ is any of the two roots of $b+b^{-1}={\rm tr}_h(\rho)$, (note that $[a,b]=[a,b^{-1}]$). If ${\rm tr}_v(\rho)\notin\{\pm 2\}$, then by (iii) there exists a unique $[a,b]\in\mathcal{U}$ such that $t([a,b])=({\rm tr}_v(\rho),{\rm tr}_h(\rho))$ and $\rho^e=-\rho^wS_{[a,b]}$; writing $\rho^{\rm sw}=-a\rho^{\rm nw}+s(a)X$, we have
$$\rho^e=(-b^{-1}\rho^{\rm nw}-s(b)X,ab^{-1}\rho^{\rm nw}+s(ba^{-1})X),$$
implying $\rho^s=-a\rho^nS_{[b,a]}$ straightforward.

(ii) Assume $\rho^n$ is NR$_\ell$; the proof for the other case is similar. By (iii), ${\rm tr}_v(\rho)=2a$ with $a\in\{\pm 2\}$.

If $\rho_1^w$, $\rho_1^e$, $\rho_2^e$ are all R, then $\rho^s=-a\rho^n=-a\rho^n C_b(0)$, and $\rho^s$ is NR$_\ell$.

If at least one of $\rho_1^w$, $\rho_1^e$, $\rho_2^e$ is NR$_{\ell'}$, then by $H_{\rm i}(T_1), H_{\rm i}(T_2)$, all of them are NR$_{\ell'}$. By (i), ${\rm tr}_h(\rho)=2b$ with $b\in\{\pm 1\}$ and $\rho^e=-b\rho^wC_a(u)$ for some $u\in\mathbb{Q}$. Since $C_a(0)=I$ but $\rho^n$ is NR, we have $u\ne 0$; by (i) again, $\ell'=\ell$ and $\rho^s$ is also NR$_\ell$.
Similarly as in the proof of Proposition \ref{prop:reducible}, we can deduce $\rho^s=-a\rho^nC_b(u^{-1})$ from $\rho^e=-b\rho^wC_a(u)$.
\end{proof}

For $T\in\mathcal{T}_{\rm ar}$ and $a,b\in\mathbb{C}^\times$, let
\begin{align}
\mathcal{R}^{a,b}(T)=\{\rho\in\mathcal{R}(T)\colon ({\rm tr}_v(\rho),{\rm tr}_h(\rho))=(t(a),t(b))\}.
\end{align}

For $\rho\in\mathcal{R}^{a,b}(T)$,
\begin{itemize}
  \item say $\rho$ is {\it (of type) VR$_0$} if $a\in\{\pm 1\}$ and $\rho^w$ is R, and write $\rho\in{\rm VR}_0$;
  \item say $\rho$ is {\it (of type) VR$_1$} if $a\notin\{\pm 1\}$, and write $\rho\in{\rm VR}_1$;
  \item say $\rho$ is {\it (of type) VNR$_\ell$} if $\rho^w$ is NR$_\ell$, and write $\rho\in{\rm VNR}_\ell$.
\end{itemize}
The three types are called {\it V-types}. Similarly, we have notions of {\it H-types}.

For $a,b\in\{\pm 1\}$, $\ell\in\{1,2\}$ and $u,v\in\mathbb{Q}$, let
\begin{align}
\mathcal{R}^{a,b}_{\rm {vn}(\ell,u)}(T)&=\{\rho\in\mathcal{R}^{a,b}(T)\colon\rho\in{\rm VNR}_\ell,\ \rho^e=-b\rho^wC_a(u)\}; \\
\mathcal{R}^{a,b}_{{\rm hn}(\ell,v)}(T)&=\{\rho\in\mathcal{R}^{a,b}(T)\colon\rho\in{\rm HNR}_\ell,\ \rho^s=-a\rho^nC_b(v)\}.
\end{align}
The proof of Theorem \ref{thm:key} clarifies that, if $u\ne 0$, then
$$\mathcal{R}^{a,b}_{{\rm vn}(\ell,u)}(T)=\mathcal{R}^{a,b}_{{\rm hn}(\ell,u^{-1})}(T).$$
Note that the involution defined in Remark \ref{rmk:involution} interchanges $\mathcal{R}^{a,b}_{\rm {vn}(\ell,u)}(T)$  with $\mathcal{R}^{a,b}_{\rm {vn}(\ell',u)}(T)$ for $\{\ell,\ell'\}=\{1,2\}$.

For $a,b\in\{\pm1\}$ and $U,V\in{\rm SL}^0(2,\mathbb{C})$, set
\begin{align}
\mathcal{R}_{{\rm vr}(a)}^{U,V}(T)&=\{\rho\in\mathcal{R}(T)\colon \rho\in{\rm VR}_0, \ \rho^n=(U,V)\}, \\
\mathcal{R}_{{\rm hr}(b)}^{U,V}(T)&=\{\rho\in\mathcal{R}(T)\colon \rho\in{\rm HR}_0, \ \rho^w=(U,V)\}.
\end{align}

From the proof of Theorem we see that the partial operation $+$ (defined in Section 3) restricts to
\begin{align*}
\mathcal{R}^{a,b_1}(T_1)&\times_h\mathcal{R}^{a,b_2}(T_2)\to\mathcal{R}^{a,b_1b_2}(T_1+T_2), &&   a\notin\{\pm 1\}, \\
\mathcal{R}_{{\rm vr}(a)}^{U,W}(T_1)&\times\mathcal{R}_{{\rm vr}(a)}^{-W,V}(T_2)\to\mathcal{R}_{{\rm vr}(a)}^{U,V}(T_1+T_2), && a\in\{\pm 1\}, \\
\mathcal{R}^{a,b_1}_{{\rm vn}(\ell,u_1)}(T_1)&\times_h\mathcal{R}^{a,b_2}_{{\rm vn}(\ell,u_2)}(T_2)\to
\mathcal{R}^{a,b_1b_2}_{{\rm vn}(\ell,u_1+u_2)}(T_1+T_2), && a,b_1,b_2\in\{\pm 1\}.
\end{align*}
The situation for $\ast$ is similar.

Here is our second main result.

\begin{thm} \label{thm:method}
One may inductively determine representations of an arborescent tangle by the following rules:

{\rm (a)} When $a\notin\{\pm 1\}$,
\begin{align}
\mathcal{R}^{a,b}(T_1+T_2)=\bigsqcup\limits_{b_1b_2=b}\mathcal{R}^{a,b_1}(T_1)+\mathcal{R}^{a,b_2}(T_2);   \label{eq:VR1}
\end{align}
when $a\in\{\pm 1\}$,
\begin{align}
\mathcal{R}^{U,V}_{{\rm vr}(a)}(T_1+T_2)&=\bigsqcup\limits_{W}\mathcal{R}^{U,W}_{{\rm vr}(a)}(T_1)+\mathcal{R}^{-W,V}_{{\rm vr}(a)}(T_2),  \label{eq:VR0} \\
\mathcal{R}_{{\rm vn}(\ell,u)}^{a,b}(T_1+T_2)&=\bigsqcup\limits_{u_1+u_2=u\atop b_1b_2=b}
\mathcal{R}_{{\rm vn}(\ell,u_1)}^{a,b_1}(T_1)+\mathcal{R}_{{\rm vn}(\ell,u_2)}^{a,b_2}(T_2).
\end{align}

For $\rho\in\mathcal{R}^{U,V}_{{\rm vr}(a)}(T_1+T_2)$, denoting $\rho_j=\rho|_{T_j}$, there are three cases:
\begin{enumerate}
  \item[\rm(i)] if $\rho\in{\rm HR}_0$, then $b_2\in\{bb_1^{\pm1}\}$ and $\rho_1,\rho_2$ are of the same H-type;
  \item[\rm(ii)] if $\rho\in{\rm HR}_1$, then either $\rho_1\in{\rm HR}_1$, or $\rho_2\in{\rm HR}_1$, or $\rho_j\in{\rm HNR}_{\ell_j}$, $j=1,2$, with $\ell_1\ne\ell_2$;
  \item[\rm(iii)] if $\rho\in{\rm HNR}_{\ell}$, then either $\rho_j\in{\rm HR}_1, \rho_{j'}\notin{\rm HR}_0$ for $j'\ne j$, or $\rho_j\in{\rm HNR}_\ell$,  $\rho_{j'}\notin{\rm HNR}_{\ell'}$ for $j'\ne j$, with $\ell'\ne\ell$.
\end{enumerate}

{\rm(b)} When $b\notin\{\pm 1\}$,
\begin{align}
\mathcal{R}^{a,b}(T_1\ast T_2)=\bigsqcup\limits_{a_1a_2=a}\mathcal{R}^{a_1,b}(T_1)\ast\mathcal{R}^{a_2,b}(T_2);   \label{eq:HR1}
\end{align}
when $b\in\{\pm 1\}$,
\begin{align}
\mathcal{R}_{{\rm hr}(b)}^{U,V}(T_1\ast T_2)&=\bigsqcup\limits_{W}\mathcal{R}_{{\rm hr}(b)}^{U,W}(T_1)\ast
\mathcal{R}_{{\rm hr}(b)}^{-W,V}(T_2),  \label{eq:HR0}\\
\mathcal{R}_{{\rm hn}(\ell,v)}^{a,b}(T_1\ast T_2)&=\bigsqcup\limits_{v_1+v_2=v\atop a_1a_2=a}
\mathcal{R}_{{\rm hn}(\ell,v_1)}^{a_1,b}(T_1)\ast\mathcal{R}_{{\rm hn}(\ell,v_2)}^{a_2,b}(T_2).
\end{align}

For $\rho\in\mathcal{R}_{{\rm hr}(b)}^{U,V}(T_1\ast T_2)$, denoting $\rho_j=\rho|_{T_j}$, there are three cases:
\begin{enumerate}
  \item[\rm(i)] if $\rho\in{\rm VR}_0$, then $a_2\in\{aa_1^{\pm1}\}$ and $\rho_1,\rho_2$ are of the same V-type;
  \item[\rm(ii)] if $\rho\in{\rm VR}_1$, then either $\rho_1\in{\rm VR}_1$, or $\rho_2\in{\rm VR}_1$, or $\rho_j\in{\rm VNR}_{\ell_j}$, $j=1,2$, with $\ell_1\ne\ell_2$;
  \item[\rm(iii)] if $\rho\in{\rm VNR}_{\ell}$, then either $\rho_j\in{\rm VR}_1, \rho_{j'}\notin{\rm VR}_0$ for $j'\ne j$, or $\rho_j\in{\rm VNR}_\ell$,  $\rho_{j'}\notin{\rm VNR}_{\ell'}$ for $j'\ne j$, with $\ell'\ne\ell$.
\end{enumerate}
\end{thm}

\begin{proof}
We only prove (a); the proof for (b) is similar.

The equations (\ref{eq:VR1}) and (\ref{eq:VR0}) are obvious. In the three cases for $\rho\in\mathcal{R}^{U,V}_{{\rm vr}(a)}(T_1+T_2)$,
(i) is obvious, (ii) and (iii) are consequences of the following:
\begin{itemize}
  \item If $\rho_j\in{\rm HR}_0$ then $\rho_1+\rho_2$ has the same H-type as $\rho_{j'}$, $j'\ne j$. This is obvious.
  \item If $\rho_1,\rho_2\in{\rm HNR}_{\ell_j}$, then
        $$\rho_1+\rho_2\in\begin{cases} {\rm HR}_0\sqcup{\rm HNR}_\ell, &\ell_1=\ell_2=\ell, \\ {\rm HR}_1,&\ell_1\ne\ell_2. \end{cases}.$$
        To see this, just assume $\rho_1^{\rm ne}=A$, then $\rho_1^{\rm nw}=-b_1B_{\ell_1}(c_1)$ for some $c_1\ne 0$, and $\rho_2^{\rm nw}=b_2B_{\ell_2}(c_2)$ for some $c_2\ne 0$, so that
        ${\rm tr}(\rho_1^{\rm nw}\rho_2^{\rm ne})\in\{\pm 2\}$ if $\ell_1=\ell_2$ and ${\rm tr}(\rho_1^{\rm nw}\rho_2^{\rm ne})\notin\{\pm 2\}$ if $\ell_1\ne\ell_2$.
\end{itemize}
\end{proof}

\begin{exmp} \label{eg:rational}
\rm

\begin{figure}[h]
  \centering
  \includegraphics[width=0.45\textwidth]{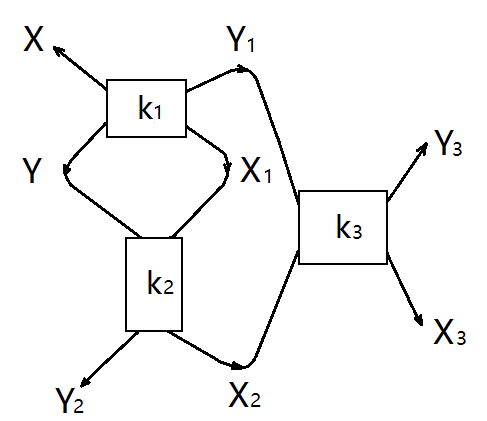}\\
  \caption{A representation of a rational tangle} \label{fig:rational}
\end{figure}

Let $T=[p/q]=[[k_1],\ldots,[k_m]]\ne[0]$.  As illustrated by Figure \ref{fig:rational}, a representation $\rho$ is determined by $(X,Y)$ (called the {\it generating pair}). More precisely, $X_j,Y_j,j=1,\ldots,m$ are all linear combinations of $X,Y$, the coefficients being polynomials in ${\rm tr}(XY)$ that are independent of the type (i.e., H-type and V-type) of $\rho$.
Suppose ${\rm tr}(XY)=t(c)$. Up to conjugacy, $\rho$ is determined by its type and $t(c)$.

If $c\notin\{\pm1\}$, then by induction on $m$, we can show
\begin{align}
\chi(\rho)=[(-1)^{\tilde{q}-1}c^q,(-1)^{\tilde{p}-1}c^p],
\end{align}
where $\tilde{p}/\tilde{q}=[[k_2,\ldots,k_m]]^{(-1)^{m-1}}$ so that $p\tilde{p}-q\tilde{q}=1$; see Section 2 of \cite{Ch16} for more details.
Let $a=(-1)^{\tilde{q}-1}c^q$, $b=(-1)^{\tilde{p}-1}c^p$, so that $[a,b]=\chi(\rho)$, and
\begin{align}
(-a)^p+(-b)^{q}=0.  \label{eq:ab-rational}
\end{align}
Note that $c=b^{\tilde{p}}a^{-\tilde{q}}$, so $a\notin\{\pm 1\}$ or $b\notin\{\pm 1\}$, and $\rho^n$, $\rho^w$, $\rho^s$, $\rho^e$ are all R.

If $c\in\{\pm 1\}$ and $Y=-cX$, then $\rho$ is abelian, so $\rho^n$, $\rho^w$, $\rho^s$, $\rho^e$ are all R.

If $c\in\{\pm 1\}$ and $Y\ne -cX$, then $(X,Y)\sim(A,-cB_\ell)$ with $\ell\in\{\pm 1\}$, and $\rho$ is reducible. Applying Proposition \ref{prop:reducible}, we can successively show that $(X,Y_j)$ is NR$_\ell$ for $j=1,2,\ldots,m$.
Consequently, $\rho^n$, $\rho^w$, $\rho^s$, $\rho^e$ are all NR$_\ell$.
\end{exmp}

\begin{rmk} \label{rmk:rational}
\rm It is worth emphasizing three features of representations $\rho$ of a rational tangle different from $[0]$:
\begin{enumerate}
  \item[\rm(i)] $\rho^n,\rho^w,\rho^s,\rho^e$ are simultaneously R or NR, in which case we simply say that $\rho$ is R or NR, respectively;
  \item[\rm(ii)] the condition (\ref{eq:ab-rational}) is also true when $\rho$ is NR, since the expressions of $X_j,Y_j$ in terms of $X,Y$ and $t(c)$ do not depend on the type of $\rho$;
  \item[\rm(iii)] $\rho$ is determined by $\rho^{\rm nw}, \rho^{\rm ne}, \rho^{\rm sw}, \rho^{\rm se}$.
\end{enumerate}
\end{rmk}

\begin{exmp} \label{eg:2-rational}
\rm Let $T=[p_1/q_1]+[p_2/q_2]$ and $\rho\in\mathcal{R}(T)$. Up to conjugacy we may assume $\rho^{\rm nw}=A$.
Let $\rho_j=\rho|_{[p_j/q_j]}$. Suppose $\chi(\rho)=[a,b]$, and $\chi(\rho_j)=[a,b_j], j=1,2$. Then by (\ref{eq:ab-rational}),
$$(-a)^{p_j}+(-b_j)^{q_j}=0, \qquad j=1,2.$$

If $\rho\in{\rm VR}_1$, then $b_1b_2=b$; up to conjugacy, $\rho^{\rm sw}=-A_1(\alpha)$, and $\rho$ is determined by
$$\rho_1^e=-(A,-aB_\ell)S_{(a,b_1)}, \qquad \rho_2^e=-(A,-aB_\ell)S_{(a,b)}.$$

If $\rho\in{\rm VNR}_\ell$, then $b_1b_2=b$; up to conjugacy, $\rho^{\rm sw}=-aB_\ell$, and $\rho$ is determined by
$$\rho_1^e=-b_1(A,-aB_\ell)C_{a}(p_1/q_1), \qquad \rho_2^e=-b(A,-aB_\ell)C_{a}(p_1/q_1+p_2/q_2).$$

If $\rho\in{\rm VR}_0$, then $\rho^{\rm sw}=-aA$, and $\rho_j^{\rm se}=-a\rho_j^{\rm ne}$, $j=1,2$, so $\rho$ is determined by $\rho_1^{\rm ne}$ and $\rho_2^{\rm ne}$. There are three possibilities:
\begin{enumerate}
  \item[\rm(i)] If $\rho\in{\rm HR}_0$, then $b_2\in\{bb_1^{\pm1}\}$, $\rho_2^{\rm ne}=-bA$, and $\rho_1,\rho_2$ are both R, hence up to conjugacy, $\rho_1^{\rm ne}=-A_1(b_1)$.
  \item[\rm(ii)] If $\rho\in{\rm HR}_1$, then up to conjugacy, $\rho_2^{\rm ne}=-A_1(b)$, and one of the following occurs:
       \begin{itemize}
           \item $\rho_1\in{\rm HR}_1,\rho_2\in{\rm HR}_0$, $b_2\in\{\pm 1\}$, $b=b_1b_2$, and $\rho_1^{\rm ne}=-A_1(b_1)$;
           \item $\rho_1\in{\rm HR}_0,\rho_2\in{\rm HR}_1$, $b_1\in\{\pm 1\}$, $b=b_1b_2$, and $\rho_1^{\rm ne}=-b_1 A$;
           \item $\rho_1\in{\rm HR}_1,\rho_2\in{\rm HR}_1$, $b_1,b_2\notin\{\pm 1\}$; up to conjugacy, $\rho_1^{\rm ne}=-A_c(b_1)$, with $\gamma_{b,b_1}(c)=-t(b_2)$.
       \end{itemize}
  \item[\rm(iii)] If $\rho\in{\rm HNR}_\ell$, then $\rho_1,\rho_2\in{\rm HR}_1$, so that $b_j\notin\{\pm 1\}$; up to conjugacy, $\rho_2^{\rm ne}=-bB_{\ell}$ and $\rho_1^{\rm ne}=-A_c(b_1)$, with $\delta_{b_1,\ell}(c)=-t(b_2)b$.
\end{enumerate}

\end{exmp}

\section{Finding all representations of an arborescent link}

The method is straightforward.
For $T\in\mathcal{T}_{\rm ar}$, a representation of $N(T)$ is the same as $\rho\in\mathcal{R}(T)$ with $\rho^e=-\rho^w$, which is equivalent to
\begin{align}
\begin{cases}
b=1, &\text{if}\  \rho\in\mathcal{R}^{a,b}(T) \ \text{with}\ a\notin\{\pm 1\}, \\
b=1, U=-V, &\text{if}\ \rho\in\mathcal{R}^{U,V}_{{\rm vr}(a)}(T) \ \text{with}\ a\in\{\pm 1\}, \\
b=1,u=0, &\text{if}\ \rho\in\mathcal{R}^{a,b}_{{\rm vn}(\ell,u)}(T);
\end{cases}
\end{align}
a representation of $D(T)$ is the same as $\rho\in\mathcal{R}(T)$ with $\rho^s=-\rho^n$, which is equivalent to
\begin{align}
\begin{cases}
a=1, &\text{if}\ \rho\in\mathcal{R}^{a,b}(T)\ \text{with}\  b\notin\{\pm 1\}, \\
a=1, U=-V, &\text{if}\ \rho\in\mathcal{R}^{U,V}_{{\rm hr}(b)}(T)\ \text{with}\  b\in\{\pm 1\}, \\
a=1,v=0, &\text{if}\ \rho\in\mathcal{R}^{a,b}_{{\rm hn}(\ell,v)}(T).
\end{cases}
\end{align}

Thus the problem is reduced to finding all representations of $T$ with some specified conditions.

\subsection{A class of 3-bridge arborescent links}

3-bridge arborescent links were completely classified by Jang \cite{Ja11}. According to \cite{Ja11} Theorem 1, there are three families of non-Montesinos 3-bridge arborescent links, one of which consists of links of the form $L=N(T)$ with
$$T=([p_1/q_1]\ast[p'_1/q'_1])+[1/k_0]+([p_2/q_2]\ast[p'_2/q'_2]).$$
Figure \ref{fig:3-bridge2} illustrates $L$ together with some directed arcs $x_0,\ldots,x_9$.

Given a representation $\tilde{\rho}$ of $L$, let $X_j=\tilde{\rho}(x_j), j=0,\ldots,9$, and let $\rho$ denote the corresponding representation of $T$. Then $\rho^e=-\rho^w$. By (iii) in Remark \ref{rmk:rational}, to describe $\rho$ (and $\tilde{\rho}$), it suffices to write down the $X_j$'s.
Up to conjugacy, we may assume $X_0=A$.

Let $\rho_0=\rho|_{[1/k_0]}$, and $\rho_j=\rho|_{[p_j/q_j]\ast[p'_j/q'_j]}$, $j=1,2$. Suppose
\begin{align}
\chi(\rho)&=[a,b], \qquad \chi(\rho_0)=[a,b_0], \\
\chi(\rho|_{[p_j/q_j]})&=[a_j,b_j], \quad \chi(\rho|_{[p'_j/q'_j]})=[a'_j,b_j], \quad j=1,2.
\end{align}
Then
\begin{align}
a=(-b_0)^{k_0}; \qquad (-a_j)^{p_j}+(-b_j)^{q_j}=(-a'_j)^{p'_j}+(-b_j)^{q'_j}=0, \quad j=1,2.
\end{align}

\begin{figure} [h]
  \centering
  \includegraphics[width=0.5\textwidth]{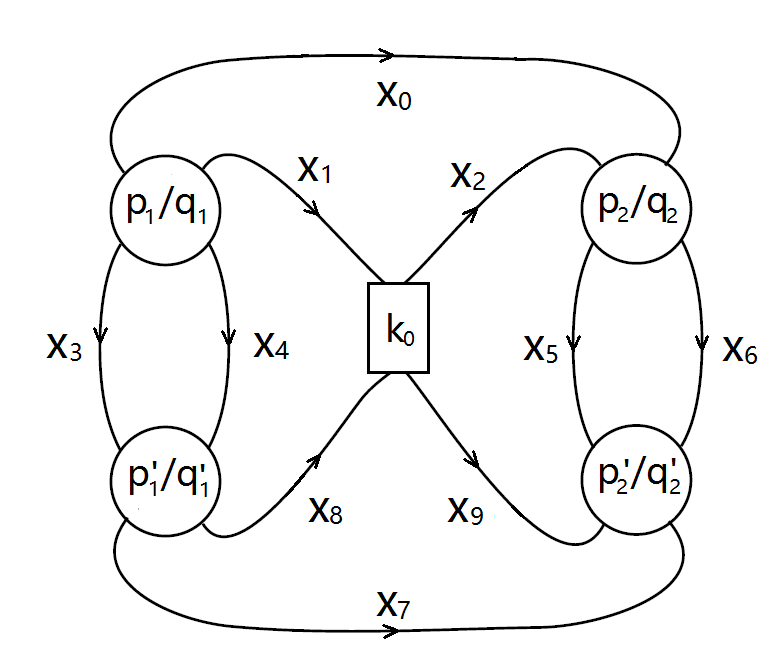}\\
  \caption{$L=N(T)$ with $T=([p_1/q_1]\ast[p'_1/q'_1])+[1/k_0]+([p_2/q_2]\ast[p'_2/q'_2])$} \label{fig:3-bridge2}
\end{figure}

To simplify the writing, let
\begin{align}
h_j=\frac{q_j}{p_j}+\frac{q'_j}{p'_j}, \qquad j=1,2.
\end{align}

\subsubsection{$\rho$ is VR$_1$}

In this case, $a\notin\{\pm 1\}$, $b_0b_1b_2=1$, and $a_ja'_j=a$ if $b_j\notin\{\pm 1\}$.
Up to conjugacy we may assume $X_7=-A_1(a)$, then
\begin{align*}
(X_1,X_8)&=-(X_0,X_7)S_{(a,b_1)},&&\ \\
(X_2,X_{9})&=(X_1,X_8)S_{(a,b_0)}, &&\ \\
(X_3,X_4)&=-(X_0,X_1)S_{(b_1,a_1)} && \text{when}\qquad b_1\notin\{\pm 1\}, \\
(X_5,X_6)&=(X_2,X_0)S_{(b_2,a_2)} && \text{when} \qquad b_2\notin\{\pm 1\};
\end{align*}
when $b_1\in\{\pm 1\}$, $X_3=-A_{c_1}(a_1)$, with $\gamma_{a,a_1}(c_1)=t(a'_1)$, and $X_{4}=-b_1X_3$;
when $b_2\in\{\pm 1\}$, $X_{6}=-A_{c_2}(a_2)$, with $\gamma_{a,a_2}(c_2)=-t(a'_2)$, and $X_5=-b_2X_6$.

\subsubsection{$\rho$ is VR$_0$}

In this case, $a\in\{\pm 1\}$, $\rho_0$ is R, and $X_j=-aX_{j-7}$ for $j=7,8,9$.

For $j=1,2$,
\begin{itemize}
  \item if $\rho_j\in{\rm HR}_0$, then $a'_j\in\{aa_j^{\pm 1}\}$, $X_{3j}=(-1)^jA_{c_j}(a_j)$, and $X_{3j-(-1)^j}=-b_jX_{3j}$;
  \item if $\rho_j\in{\rm HR}_1$, then $a_ja'_j=a$, and
        $$\begin{cases} (X_3,X_4)=-(X_0,X_1)S_{(b_1,a_1)}, & j=1, \\ (X_5,X_6)=(X_2,X_0)S_{(b_2,a_2)}, &j=2; \end{cases}$$
  \item if $\rho_j\in{\rm HNR}_{\ell_j}$, then $a_ja'_j=a$, $h_j=0$, and
        $$\begin{cases} (X_3,X_4)=-a_1(X_0,X_1)C_{b_1}(q_1/p_1), & j=1, \\ (X_5,X_6)=a_2(X_2,X_0)C_{b_2}(q_2/p_2), &j=2. \end{cases}$$
\end{itemize}

To determine $\rho$, it remains to work out $X_1,X_2$. There are 7 cases, as listed below.
In each case, (up to conjugacy) the value of $(X_1,X_2)$ is given in the upper cell, and an extra constraint is given in the lower cell.

\begin{center}
\begin{tabular}{|c|c|c|c|}
  \hline
  \backslashbox{$\rho_1$}{$\rho_2$} & HR$_0$ & HR$_1$ & HNR$_{\ell_2}$ \\
  \hline
  \multirow{2}{*}{HR$_0$} & $(-b_1A,-b_2A)$  & $(-b_1A,-A_1(b_2))$ & \ \\
  \cline{2-3}
      & $b_1b_2=b_0$ & $t(b_2)b_1=t(b_0)$ & \ \\
  \hline
  \multirow{2}{*}{HR$_1$} & $(-A_1(b_1),-b_2A)$ & $(-A_1(b_1),-A_c(b_2))$ & $(-A_c(b_1),-b_2B_{\ell_2})$ \\
  \cline{2-4}
      & $t(b_1)b_2=t(b_0)$ & $\gamma_{b_1,b_2}(c)=-t(b_0)$ & $\delta_{b_1,\ell_2}(c)b_2=-t(b_0)$ \\
  \hline
  \multirow{2}{*}{HNR$_{\ell_1}$} & \ & $(-b_1B_{\ell_1},-A_c(b_2))$ & $(-b_1B_{\ell_1},-b_2B_{\ell_2}(d))$ \\
  \cline{3-4}
      & \ & $\delta_{b_2,\ell_1}(c)b_1=-t(b_0)$ & $\star$ \\
  \hline
\end{tabular}
\end{center}
The $\star$ stands for the condition
$$t(b_0)=\begin{cases} (2-d)b_1b_2, &\ell_1\ne\ell_2, \\ 2b_1b_2, & \ell_1=\ell_2. \end{cases}$$

\subsubsection{$\rho$ is VNR$_\ell$}

In this case, $b_0,b_1,b_2\in\{\pm 1\}$, and $b_0b_1b_2=1$. Up to conjugacy we may assume $X_7=-aB_\ell$.

For $j=1,2$, $\rho_j$ is either HR$_0$ or HNR$_\ell$.
If $\rho_1$ and $\rho_2$ are both HR, then $X_1=-b_1A, X_2=-b_2A$ so that $\rho_0$ is abelian, which is impossible.
Thus there are three possibilities:
\begin{itemize}
  \item If $\rho_1$ is HR$_0$ and $\rho_2$ is HNR$_\ell$, then $h_2=-k_0$, $a_1,a'_1\notin\{\pm 1\}$,  and $X_j=-b_1X_{j-1}$ for $j=1,4,8$.
        Up to conjugacy, $X_3=-A_{c_1}(a_1)$, with $\delta_{a_1,\ell}(c_1)=-t(a'_1)a$;
        The remaining $X_j$'s are determined by
       \begin{align*}
       (X_2,X_9)&=b_0(X_1,X_8)C_a(1/k_0), \\
       (X_5,X_6)&=a_2(X_2,X_0)C_{b_2}(q_2/p_2).
       \end{align*}
  \item If $\rho_1$ is HNR$_\ell$ and $\rho_2$ is HR$_0$, then $h_1=-k_0$, $a_2,a'_2\notin\{\pm 1\}$, and $X_2=-b_2X_0$, $X_5=-b_2X_6$, $X_9=-b_2X_7$. Up to conjugacy, $X_6=A_{c_2}(a_2)$, with $\delta_{a_2,\ell}(c_2)=-t(a'_2)a$.
       The remaining $X_j$'s are determined by
       \begin{align*}
       (X_1,X_8)&=b_0(X_2,X_9)C_a(-1/k_0), \\
       (X_3,X_4)&=-a_1(X_0,X_1)C_{b_1}(q_1/p_1).
       \end{align*}
  \item If $\rho_1$ and $\rho_2$ are both HNR$_\ell$, then
      $$\frac{1}{k_0}+\frac{1}{h_1}+\frac{1}{h_2}=0,$$
      and the $X_j$'s are determined as follow:
      \begin{align*}
      (X_1,X_8)&=-b_1(X_0,X_7)C_a(1/h_1), \\
      (X_2,X_{9})&=-b_2(X_0,X_7)C_a(-1/h_2), \\
      (X_3,X_4)&=-a_1(X_0,X_1)C_{b_1}(q_1/p_1), \\
      (X_5,X_6)&=a_2(X_2,X_0)C_{b_2}(q_2/p_2).
      \end{align*}
\end{itemize}

\medskip

{\bf Acknowledgement}

This work is supported by NSFC-11401014 and NSFC-11501014.


\begin{thebibliography}{}


\bibitem{Ch16}
H.-M. Chen,
\textsl{Trace-free representations of Montesinos links}. \\
J. Knot Theor. Ramif. 27 (2018), no. 8, 1850050 (10 pages).


\bibitem{Ja11}
Y. Jang,
\textsl{Classification of 3-bridge arborescent links}.
Hiroshima Math. J. 41 (2011), 89--136.


\bibitem{KL03}
L.H. Kauffman, S. Lambropoulou,
\textsl{From Tangle Fractions to DNA}.
In: Topology in Molecular Biology, pp. 69--110, Springer Berlin Heidelberg, 2003.

\bibitem{KL04}
L.H. Kauffman, S. Lambropoulou,
\textsl{On the classification of rational tangles}.
Adv. Appl. Math. 33 (2004), no. 2, 199--237.


\bibitem{Le14}
S. Lewallen,
\textsl{Khovanov homology of alternating links and ${\rm SU}(2)$ representations of the link group}.
arXiv:0910.5047.

\bibitem{Li92}
X.-S. Lin,
\textsl{A knot invariant via representation spaces}.
J. Diff. Geom. 35 (1992), 337--357.



\bibitem{NY12}
F. Nagasato, Y. Yamaguchi,
\textsl{On the geometry of the slice of trace-free ${\rm SL}_{2}(\mathbb{C})$-characters of a knot group}.
Math. Ann. 354 (2012), 967--1002.



\end{thebibliography}
\end{document}